\documentclass[12pt]{amsart}
\usepackage{graphicx} % Required for inserting images
\usepackage{amsthm}
\usepackage{amssymb}
\usepackage{mathtools}
\usepackage{MnSymbol}
\usepackage{xcolor}
\usepackage{bbm}
\usepackage[normalem]{ulem}
\DeclareMathAlphabet{\mymathbb}{U}{bbold}{m}{n}
\usepackage{dsfont}
\usepackage{hyperref}
\usepackage{amssymb}
\usepackage{enumitem}
\usepackage[T1]{fontenc}
\usepackage{bbm}
\usepackage{verbatim}
\usepackage{mathrsfs}
\usepackage{dutchcal}
\DeclareMathAlphabet{\standardcal}{OMS}{cmsy}{m}{n}

\newtheorem{theorem}{Theorem}[section]
\newtheorem{corollary}[theorem]{Corollary}
\newtheorem{lemma}[theorem]{Lemma}

\newtheorem{fact}[theorem]{Proposition}

\theoremstyle{definition}
\newtheorem{definition}[theorem]{Definition}
\newtheorem{remark}[theorem]{Remark}
\newtheorem{example}[theorem]{Example}

\newcommand{\mc}{\mathcal}

\newcommand{\x}{\mathbbm{x}}
\newcommand{\y}{\mathbbm{y}}

\author{}
\title{Expansiveness of vertical subgroups of the Heisenberg group% / Vertical expansive subdynamics of actions of the Heisenberg group
}

\author[M. Prusik]{Michał Prusik}
\address{Faculty of Pure and Applied Mathematics\\ Wroclaw University of Science and Technology\\
Wybrzeże Wyspiańskiego 27\\
50-370 Wrocław Poland}
\email{michal.prusik@pwr.edu.pl}



\begin{document}
    \begin{abstract}
       %In this paper, we attempt to transfer the ideas proposed by Boyle and Lind concerning the expansiveness of directions in the acting group $\mathbb Z^d$ for the case of the actions of the Heisenberg group. We assume the action via homeomorphisms of the discrete Heisenberg group $H$ on a compact metric space $(\mathbbm X,\rho)$, and we embed this group in the (continuous) $2D+1$-dimensional Heisenberg group $\mc H$. We present the definition of expansive subsets of $\mc H$. Finally, as the main results, we study the expansiveness of vertical subgroups of the Heisenberg group. In particular, we show that, if only the space $\mathbbm X$ is infinite, the center of $\mc H$ cannot be expansive, and that there always exists at least one nonexpansive $2D$-dimensional vertical subgroup.
%
    In the paper we study expansiveness along distinguished subsets in the case of a continuous action of the discrete Heisenberg group on a compact metric space $(\mathbbm X,\rho)$. Transferring the ideas proposed by Boyle and Lind for continuous actions of $\mathbb{Z}^D$, we embed the acting group in the (continuous) $(2D+1)$-dimensional Heisenberg group $\mc H$ and define expansive subsets of $\mc H$. We focus on the expansiveness of vertical subgroups of the Heisenberg group. In particular, we show that, if only the space $\mathbbm X$ is infinite, the center of $\mc H$ cannot be expansive, and that there always exists at least one nonexpansive $2D$-dimensional vertical subgroup.
    \end{abstract}
	
	\maketitle	
    \section{Introduction}

    In the late 1980s, Milnor proposed the study of multidimensional cellular automata by analyzing their dynamical behavior when restricted to distinguished \emph{directions}. In his paper \cite{M}, the lattice $\mathbb Z^D$, on which an automaton is defined, is embedded in the space $\mathbb R^D$, and the role of the directions is played by linear subspaces of this continuous space, regardless of whether they intersect $\mathbb Z^D$ non-trivially or not. Milnor introduced and studied the notion of \emph{directional entropy} measured along such subspaces. In 1997, his work was extended by Boyle and Lind \cite{BL} to a continuous action of $\mathbb Z^D$ on an infinite compact metric space. They focused on the notion of expansiveness of the directions and defined \emph{expansive components}, which are connected components of expansive directions of a fixed dimension. They showed that several dynamical properties of the directions, including Milnor's directional entropy, vary nicely within these components, and a ``phase transition'' can be observed between them. 
    
    One of the main results given by Boyle and Lind is that if only the underlying space $\mathbbm X$ is infinite, there always exists at least one $D-1$-dimensional nonexpansive subspace, and that the set of $k$-dimensional nonexpansive subspaces $\mathbb N_k$ is compact in the Grassmann manifold $\mathbb G_k$ (equivalently, that the set of expansive subspaces $\mathbb E_k = \mathbb G_k\setminus \mathbb N_k$ is open), for any dimension $k<D$. For the case of $D=2$, they also constructed examples realizing any compact subset of $\mathbb G_1$ as $\mathbb N_1$, with the exception of a singleton containing one irrational line, which was completed later by Hochman \cite{H}. At the end of \cite{BL}, the authors left a list of open problems, which they end with a question: ``(...) we have been considering the lattice $\mathbb Z^d$ in $\mathbb R^d$, and considering the subdynamics of closed subgroups in $\mathbb R^d$. What generalizes to lattices in Lie groups?'' Motivated by this remark, we aim to study subdynamics of actions of more general groups.  Let us recall that an action of a group $G$ on a (topological) space $\mathbbm X$ is a map $\varphi: G\times \mathbbm X \to \mathbbm X$ such that $\varphi(g_1g_2,\cdot) = \varphi(g_1,\cdot)\circ \varphi(g_2,\cdot)$ and $\varphi(g^{-1},\cdot) = \varphi(g,\cdot)^{-1}$ hold for each $g,g_1,g_2\in G$. $G$ \emph{acts via homeomorphisms} if for each $g\in G$ its action $\varphi(g,\cdot)\in \text{Homeo}(\mathbbm X)$. It seems reasonable to start the investigation with the case of the $2D+1$-dimensional Heisenberg group. We generalize some of the results from \cite{BL} for general subsets of this group, and then provide a detailed analysis of the expansiveness of \emph{vertical groups}, which are the closest to the euclidean subspaces of $\mathbb R^D$. In particular, we show that if the space $\mathbbm X$ is infinite, the center of the group is never expansive (Theorem \ref{Z_never_expansive}), and that there is always a $2D$-dimensional nonexpansive vertical group (Theorem \ref{nonexp2D}).

    The article is organized as follows: the following two sections provide basic information about the Heisenberg group, its structure, and its topology. We suspect that most of the presented facts are well known, but some of them were hard to find in the literature, at least in the language presented here. Therefore, we provide their proofs to keep the paper self-contained. Then we move on to studying the action of the discrete Heisenberg group $H$ on a compact metric space $(\mathbbm X, \rho)$. We assume that $H$ acts on $\mathbbm X$ via homeomorphisms. The fourth section starts the general analysis of the expansiveness of the subsets of the Heisenberg group. In the last section, we provide the main results concerning the vertical groups.
    
    I want to thank my PhD supervisor, Bartosz Frej, for his guidance during my work and for helping me with the preparation of this paper. 

	\section{The Heisenberg Group}
	
	Let $D$ be a positive integer. Let $\{e_i, f_i\}_{i=1}^{D}$ be the standard basis in $\mathbb R^{2D}$. For $v\in\mathbb R^{2D}$, by $p_i, q_i$ we denote its coordinates in this basis, i.e., $v = \sum_{i=1}^{D} (p_i e_i + q_i f_i)$. We define the skew-symmetric bilinear form $\omega:\mathbb R^{2D} \to \mathbb R$ as follows:
	\begin{equation*}
		\omega(v,\tilde {v}) = \sum_{i=1}^{D} (p_i \tilde q_i - q_i \tilde p_i) \text{, for each } v,\tilde v\in \mathbb R^{2D},
	\end{equation*}
	where $p_i,q_i$ and $\tilde p_i, \tilde q_i$ are the coordinates of $v$ and $\tilde v$, respectively. For each $i,j = 1,\ldots,D$ we have $\omega(e_i,e_j) = \omega(f_i,f_j) = 0$ and $\omega (e_i,f_j) = \delta_{ij}$, where $\delta_{ij}$ is the Kronecker delta.
	\begin{definition}
	    \emph{The Heisenberg group} is the pair $(\mc H, \cdot)$, where $\mc H=\{(v,u): v\in\mathbb R^{2D}, u\in \mathbb R\}$ and the group operation is given by:
        \begin{equation*}
		(v\,u) \cdot (\tilde v, \tilde u) = (v + \tilde v, u + \tilde u + \frac12 \omega(v, \tilde v)).
	\end{equation*}
	\end{definition}
    We denote its neutral element $(0,0)$ by $\mathbf 0$. This group is often represented (isomorphically) as the set of all upper triangular matrices with the operation of matrix multiplication. We will thus stick to multiplicative notation for the group, especially since the operation is not commutative. On the other hand, let us note that
	$(v,u)^{-1} = (-v,-u)$ and $(v,u)^n = (nv, nu)$  for each $n \in\mathbb N$, 
	hence 
    \begin{equation}    \label{eq:additive}
    (v,u)^n = (nv,nu) \qquad \textrm{for every } n\in\mathbb Z.
    \end{equation}
    In $\mc H$ we distinguish the following elements: $x_i = (e_i, 0)$ and $y_i = (f_i, 0)$, for $i=1,\ldots ,D$, and  $z = (0, 1)$. For each $i=1,\ldots,D$ we have $z = x_iy_ix_i^{-1}y_i^{-1}$. Moreover, for each $i\neq j$, we have $x_i x_j = x_j x_i$, $y_i y_j = y_j y_i$, and $x_iy_j = y_jx_i$. In general, for any $g=(v_g, u_g)$ and $h=(v_h, u_h)$, we have the formula
    \begin{equation}
        \label{eq:gh=hg(0,omega)}
         gh = hg (0,\omega(v_g, v_h)).
    \end{equation}
	
	For $A\subset \mc H$, by $\langle A \rangle$ we denote the subgroup generated by $A$.
	\begin{definition}
		Let $H = \langle \{x_i, y_i: i=1,\ldots,D\} \rangle$. We call $H$ the \emph{discrete Heisenberg group}. 
	\end{definition}
    The discrete Heisenberg group $H$ can be explicitly written as:
    \begin{equation}
    \label{form_of_discrete_H}
        H = \left\{(v,u)\in\mc H: v = \sum_{i=1}^D (p_ie_i + q_if_i)\in\mathbb Z^{2D} \textrm{ and } u\in \mathbb Z+\frac12 \sum_{i=1}^Dp_iq_i\right\}.
    \end{equation}
    Hence, the last coordinate $u$ is either an integer or an integer plus $\frac12$, depending on the first coordinate $v$, which is always an integer vector.
	
	We call the subgroup $\mc Z = \{uz = (0,u): u\in\mathbb R\}$ the \emph{vertical axis}. It is easy to show that $\mc Z$ is the center of the Heisenberg group and that the discrete vertical axis $Z = \langle \{z\} \rangle$ is the center of the discrete group $H$.     
    In particular, $\mc Z$ and $Z$ are normal subgroups of $\mc H$ and $H$, respectively. We make the following two observations concerning $\mc Z$.
    \begin{lemma}
    \label{AZ=Z_implies_Ag=gA}
        Let $A\subset \mc H$ be such that $A\mc Z = A$. Then for any $g\in \mc H$ we have $Ag = gA$. Moreover, for any $g_1,g_2\in \mc H$ we have $Ag_1g_2 = Ag_2g_1$.
    \end{lemma}
    \begin{proof}
        By formula \eqref{eq:gh=hg(0,omega)} we have $Ag\subset gA\mc Z= gA$ and $gA\subset Ag\mc Z$. But 
        $$Ag\mc Z = A\mc Z g = Ag.$$ 
        For the second statement, notice that if $A\mc Z = A$, then by the calculation above, $Ag$ also has this property. Therefore, we have:
        \begin{equation*}
            Ag_1g_2 = g_2(Ag_1) = g_2(g_1A) = g_2g_1A = Ag_2g_1.
        \end{equation*}
    \end{proof}

	\begin{lemma}
		\label{normality_characterization}
		A subgroup $\mc G$ of $\mc H$ is normal if and only if $\mc{G}$ contains~$\mc Z$.
	\end{lemma}
	\begin{proof}
		If $\mc {Z \subset G}$, the normality follows from Lemma \ref{AZ=Z_implies_Ag=gA}.	
		On the other hand, assume that $\mc G$ is normal and that it is not $\mc Z$. Then there is $g=(v_g,u_g)\in\mc{G}$ such that $v_g\neq 0$. 
        Since $v_g$ is nonzero, for any $a\in\mathbb R$ we can find $h=(v_h,0)$ such that $\omega (v_h,v_g)=a$. Then we have $g\cdot (0,a) = (v_g,u_g+a) = hgh^{-1} \in \mc G$. Hence $g\mc Z \subset \mc G$, and therefore $\mc Z = g^{-1}g\mc Z\subset \mc G$. 
		
	\end{proof}
	
	\begin{definition}
		For $r>0$, let $\delta_r:\mc H \to \mc H$ be a map defined by:
		\begin{equation*}
			\delta_r((v,u)) = (r v, r^2 u).
		\end{equation*}
		We call $\delta_r$ the \emph{dilation by $r$}.
	\end{definition}
     For any $g,h\in\mc H$ and $r,s>0$ we have
		\[
			\delta_r(gh) = \delta_r g \cdot \delta_r h \qquad \textrm{and}\qquad
			\delta_r(\delta_s g) = \delta_{rs}g.
		\]   
    
	\begin{definition}
		We say that $A\subset \mc H$ is \emph{dilation invariant} if for any $r>0$, we have $\delta_r A = A$. A dilation invariant subgroup is called \emph{homogeneous}.
	\end{definition}
	
	\begin{fact}
		\label{nonzero_omega_cont_Z}
		If a subgroup $\mc G$ contains two elements $g_1=(v_1,u_1)$ and $g_2=(v_2,u_2)$ such that $\omega(v_1,v_2)\neq 0$, then $\mc G$ contains a non-trivial subgroup of $\mc Z$, namely $\mc \langle (0,\omega(v_1,v_2)) \rangle$.
	\end{fact}
    
	\begin{proof}
		  After simple calculation, we get $g_1g_2g_1^{-1}g_2^{-1} = (0,0,\omega(v_1,v_2))\in \mc G$.
	\end{proof}

    \begin{definition}
        A linear subspace $V\subset \mathbb R^{2D}$ is called \emph{isotropic} if for any $v_1,v_2\in V$ we have $\omega(v_1,v_2)=0$.
    \end{definition}
		
	\begin{definition}
		Let $V\subset \mathbb R^{2D}$ be a linear subspace. 
        \begin{enumerate}
            \item 
        We call a normal subgroup $\mc G\subset \mc H$ of the form $\mc G=V\times \mathbb R$ a \emph{vertical group}. 
        \item A \emph{horizontal group} is a subgroup of $\mc H$ of the form $V\times \{0\}$, where $V$ is isotropic.
        \end{enumerate}
    \end{definition}    
    Both vertical and horizontal groups are homogeneous, but horizontal groups are not normal.
    
	\begin{fact}
		\label{thm_dil_inv_characterization}
		A subgroup $\mc G$ of $\mc H$ is homogeneous if and only if it is either a horizontal or a vertical group.
	\end{fact}
  
	\begin{proof}
		Assume that $\mc G$ is homogeneous and that it is not a horizontal group. Being a homogeneous group implies that $V=\pi(\mc G)$ is a linear space, so $V\times \mathbb R$ is  a vertical group, and $\mc G\subset V\times \mathbb R$. On the other hand, $\mc G$ contains some $(v,u)$ such that $u\neq 0$. Therefore, it contains $(nv,nu)$ for any $n\in\mathbb N$. But by dilation invariance, it also contains $(nv,n^2u)$. Thus, it contains $(0,(n^2-n)u)$. Using dilation invariance once more, we get that $\mc G$ contains $\mc Z$, so $\mc G = V\times \mathbb R$.
	\end{proof}

	\begin{fact}
		If $\mc G$ is a linear subspace of $\mathbb R^{2D+1}$ such that the linear subspace $\pi(\mc G)$ is isotropic then $\mc G$ is a subgroup.
	\end{fact}
    If $\mc G$ is neither horizontal nor vertical, then we call it an \emph{inclined group}. Notice that inclined groups are not homogeneous.
	
	\begin{fact}
		A linear subspace of $\mc H$ is a subgroup if and only if it is either a horizontal, a vertical or an inclined group.
	\end{fact}	
    For $A\subset \mc H$ by $\textrm{{Span}}(A)$ we denote the smallest such group containing $A$. By $\textrm{{Lin}}(A)$, we denote the smallest linear space containing $A$.
	
	\begin{lemma}
    \label{For_groups_Span_Lin_is_the_same}
		If $\mc G\subset \mc H$ is a subgroup, then $\textrm{\emph{Span}}(\mc G) = \textrm{\emph{Lin}}(\mc G)$
	\end{lemma}
	
	\begin{proof}
		Let $V=\pi(\textrm{{Lin}}(\mc G))$ (a linear subspace of $\mathbb R^{2D}$). We have three possibilities:
		\begin{itemize}
			\item[\emph{i.}] $\mc G\subset V\times\{0\}$. Then obviously $\textrm{{Span}}(\mc G) = \textrm{{Lin}}(\mc G)=V\times \{0\}$.
			\item[\emph{ii.}] There exist $(v_1,u_1),(v_2,u_2)\in\mc G$ such that $\omega(v_1,v_2)\neq 0$. Then, by Proposition \ref{nonzero_omega_cont_Z}, $\mc G$ contains $\langle(0,\gamma)\rangle$ for some $\gamma\neq 0$, thus $\textrm{Span}(\mc G)=\textrm{Lin}(\mc G) = V\times \mathbb R$.
			\item[\emph{iii.}] $\mc G\not\subset V\times\{0\} $, but for each $(v_1,u_1),(v_2,u_2)\in\mc G$ we have $\omega(v_1,v_2)= 0$. Then $\textrm{Lin}(\mc G)$ is either an inclined group or $V\times \mathbb R$. 
		\end{itemize}
    \end{proof}

%%%%%%%%%%%%%%%%%%%%%%%%%%%%%%%%%%%%%%%%%%%%%%%%%%%%%%%%%%%%%%%%%%%%%%%%%%%%%%%%%%%%%%%%%%%%%%%%%%%%%%%%%%%%%%%%%%%%%%%%%%%%%%%%%%%%%%%%%%%%%%%%%%%%%%%%%%%%%%%%%%%%
	\section{The Topology of the Heisenberg Group}
		
	We consider the Euclidean topology $\tau_e$ on $\mc H$, that is the natural topology on $\mathbb R^{2D+1}$. By $d_e$ we denote the Euclidean metric on $\mc H$ and we write $|v|$ for the Euclidean norm of $v$. Note that, for every $g\in \mc H$, the self-maps $h\mapsto hg$ and $h\mapsto gh$ are continuous with respect to~$\tau_e$.
    
    For the following definitions of homogeneous metric and homogeneous norm, compare, e.g., \cite{BLU} and \cite{DRES}.
    
	\begin{definition}
		We call a metric $d$ on $\mc H$ \emph{homogeneous} if it satisfies:
		\begin{enumerate}
			\item [\emph{i.}] $d(gf,hf) = d(g,h)$,
			\item [\emph{ii.}] $d(\delta_r g,\delta_r h) = r d(g,h)$.
		\end{enumerate}
	\end{definition}    
	
	We denote the topology induced by a metric $d$ by $\tau_d$.
	\begin{fact}
		Let $d$ be a homogeneous metric. Then $\tau_d=\tau_e$.
	\end{fact}
    For the proof, see \cite[Proposition 2.26]{DRES}.
    
	\begin{definition}
		We call a function $||\cdot||:\mc H\to [0,\infty)$ a \emph{homogeneous norm} if it satisfies:
		\begin{enumerate}
			\item [\emph{i.}] $||g|| = 0$ if and only if $g=\mathbf 0$,
			\item [\emph{ii.}] $||\delta_r g|| = r ||g||$.
			\item [\emph{iii.}] $||g^{-1}||=||g||$.
			\item [\emph{iv.}] $||gh||\leqslant ||g||+||h||$.
		\end{enumerate}
	\end{definition}
	\begin{fact}
		Let $||\cdot||$ be a homogeneous norm. Then $d(g,h):=||gh^{-1}||$ is a homogeneous metric. Conversely, let $d(\cdot,\cdot)$ be a homogeneous metric. Then $||g||:=d(g,\mathbf 0)$ is a homogeneous norm and $d(g,h) = ||gh^{-1}||$.
	\end{fact}
	\begin{example}
		An important example of a homogeneous norm is the \emph{Cygan-Kor\'anyi norm}: 
		\begin{equation}
			\label{C-K_norm}
			||(v,u)||_{CK} = (|v|^4 + |u|^2)^{\tfrac12}
		\end{equation}
		It is also used in a more general form, namely $(|v|^4 + c|u|^2)^{\tfrac12}$, where $c>0$. We will always refer to the form given by \ref{C-K_norm}. By $d_{CK}$ we denote the metric induced by $||\cdot||_{CK}$.
	\end{example}
	\begin{lemma}
    \label{Homog_norm_equiv_to_CK}
		Every homogeneous norm $||\cdot||$ is equivalent to $||\cdot||_{CK}$, i.e., there exists a positive number $c$ such that $\frac1c||\cdot||_{CK}\leqslant ||\cdot||\leqslant c||\cdot||_{CK}$.
	\end{lemma}
    For the proof, see e.g. \cite[Proposition 5.1.4]{BLU}.

    \begin{remark}
        \,
        \begin{enumerate}
            \item In the context of more general Carnot groups, as it is in \cite{DRES}, a notion of a \emph{quasimetric} (or a \emph{quasidistance}) is considered. The difference between a metric and a quasimetric is that a quasimetric requires only a weak triangle inequality: $d(g_1,g_2)\leqslant C(d(g_1,g_3) + d(g_3, g_2))$ for some uniform constant $C>0$. A quasimetric still induces the Euclidean topology, however, it may not be continuous as a function on $\mc H\times \mc H$. All results in this article may be written for continuous quasimetrics---we restricted to metrics to simplify the presentation.

            \item In general, for a homogeneous norm, the triangle inequality is not required---it is sufficient that the norm is continuous. In particular, Lemma \ref{Homog_norm_equiv_to_CK} holds in this case (even without the symmetry assumption). However, such norms induce only continuous quasimetrics, hence we assumed the triangle inequality.

            \item In literature, it is often required that a homogeneous metric is left invariant. We assume right invariance due to our dynamical approach---we will commonly use right translates $Ag$ of the subsets of $\mc H$.
        \end{enumerate}
    \end{remark}
    
	From now on we fix a homogeneous metric $d$ on $\mc H$, and we let $||\cdot||$ denote the homogeneous norm related to $d$.
    \begin{lemma}
    \label{distance_between_G_and_Span(G)}
        Let $\mc G$ be a subgroup of $\mc H$. There is $c>0$ such that for each $g\in\textrm{\emph{Span}}(\mc G)$ we have $d(g,\mc G)\leqslant c$.
    \end{lemma}

    \begin{proof}
        We will show this for $d = d_{CK}$. The general case follows from Lemma \ref{Homog_norm_equiv_to_CK}.
    
		Let $\{(v_1,u_1),(v_2,u_2),\ldots,(v_m,u_m)\}\subset \mc G$ be a basis of $\textrm{Lin}(\mc G) = \textrm{Span}(\mc G)$ (see Lemma \ref{For_groups_Span_Lin_is_the_same}). Fix $(v,u)\in \textrm{Span}(\mc G)$. There is a linear combination $(\tilde v, \tilde u) = \sum_{i=1}^m k_i (v_i,u_i)$ with integer coefficients, such that 
        \begin{equation*}
            |v-\tilde v|\leqslant \frac{m}{2}\max\{|v_i|\}=:c_1\text{ and }|u-\tilde u|\leqslant \frac{m}{2}\max\{|u_i|\}=:c_2.
        \end{equation*}
        Denote $g_i=(v_i,u_i)$.        

        Now we consider two cases. If the linear space $V=\pi(\textrm{Span}(\mc G))$ is isotropic, i.e., $\omega(v_i,v_j) = 0$ for each $i,j$, then $(\tilde v, \tilde u)={g_1}^{k_1}\dots{g_n}^{k_n}\in\mc G$ by \eqref{eq:additive}. Moreover,
        \begin{equation*}
            d_{CK}((v,u),(\tilde v, \tilde u)) = (|v-\tilde v|^{4} + |u-\tilde u|^2)^{\frac14}\leqslant (c_1^{4} + c_2^2)^{\frac14} =:c.
        \end{equation*}
        
        %If $V$ is not isotropic, then instead of $(\tilde v, \tilde u)$, we have to take $(\tilde v, \hat u) =\prod_{i=1}^m(v_i,u_i)^{k_i}\in\mc G$, where $\hat u = \sum_{i=1}^m k_iu_i + \frac12\sum_{i<j}k_ik_j\omega(v_i, v_j)$. However, by Proposition \ref{nonzero_omega_cont_Z}, there is an element $(0,\gamma)\in \mc G$, where $\gamma \neq 0$. Take $l\in\mathbb Z$ such that $|u-\hat u -\frac12 \omega(v,\tilde v) - l\gamma|\leqslant \frac{|\gamma|}{2} =: c_3$. We have $(\tilde v, \hat u+ l\gamma) = (\tilde v, \hat u)(0,\gamma)^l\in\mc G$ and 
        %\begin{equation*}
        %    d_{CK}((v,u),(\tilde v, \hat u+ l\gamma)) = (|v-\tilde v|^{4} + |u-\tilde u-\frac12 \omega(v,\tilde v) - l\gamma|^2)^{\frac14}\leqslant (c_1^{4} + c_3^2)^{\frac14} =:c.
        %\end{equation*}
        If $V$ is not isotropic, then $g={g_1}^{k_1}\dots{g_n}^{k_n} = (\tilde v, \tilde u + w)$, where $w = \frac12\sum_{i<j}k_ik_j\omega(v_i, v_j)$. By Proposition \ref{nonzero_omega_cont_Z}, $\mc G$ contains an element $h=(0,\gamma)$, where $\gamma > 0$. Take $l\in\mathbb Z$ such that $|w+l\gamma|\leqslant |\gamma|$. Then,
        \begin{equation*}
            d_{CK}((v,u),gh^l) = (|v-\tilde v|^{4} + |u-\tilde u -w -l \gamma|^2)^{\frac14} \leqslant (c_1^{4} + (c_2+|\gamma|)^2)^{\frac14} =:c.
        \end{equation*}
    \end{proof}
    
    \begin{corollary}
    \label{d(g,H)<=lambda}
        There exists $\lambda >0$ with the property that for any $g\in\mc H$ there is $\tilde g\in H$ such that $d(g,\tilde g)\leqslant \lambda$.
    \end{corollary}

    We write $B(r)$ for a closed ball $\{g \in \mc H:d(\mathbf{0}, g)\leqslant r\}$.
    We say that a set $A\subset \mc H$ is bounded if it is contained in $B(r)$ for some $r$.
	    
	\begin{fact}
		\label{Bounded_are_finite}
		Any bounded set $A\subset \mc H$ satisfies $|A\cap H|<\infty$.
	\end{fact}
    \begin{proof}
        Similarly as for Lemma \ref{distance_between_G_and_Span(G)}, it suffices to show the statement for the case $d=d_{CK}$. If a subset of $H$ is infinite, then it must contain a sequence $(v_n,u_n)$ such that at least one of the norms of its coordinates $|v_n|$, $|u_n|$ tends to infinity with $n$.  Therefore, $||(v_n,u_n)||_{CK}$ tends to infinity.
    \end{proof}
    
	For $A\subset \mc H$ and $t\geqslant 0$ by $A^t$ we denote the thickening of $A$ by $t$, i.e., $A^t=\{g\in \mc H:  d(g,A)\leqslant t\}$, where $ d(g,A)=\inf\{ d(g,h):h\in A\}$. The following easy observation is left without a proof.
	\begin{fact}
		\label{Thickening_the_thickening}
		We have $(A^t)^s\subset A^{t+s}$.% \M{Pomyśleć czy nie ma równości.}
	\end{fact}
    
    Throughout this article, we denote by $\pi$  the natural projection onto first $2D$ coordinates, that is, $\pi: \mc H\to \mathbb R^{2D}$ such that $\pi((v,u)) = v$. 
    
	\begin{lemma}
		\label{Thickening_vertical_group}
		Let $\mc G$ be a vertical group. Then $d_{CK}(g,\mc G) = d_e(g,\mc G)$.
	\end{lemma}
	\begin{proof}
		Notice that since $\mc Z\subset \mc G$, for $g=(v_g,u_g)$, the distance $d_e(g,\mc G)$ equals $d_e((v_g,0), \pi(\mc G)\times \{0\})$. On the other hand, for $h = (v_h,u_h)$, we have $d_{CK}(g,h) = (|v_g-v_h|^4+(u_g-u_h-\tfrac12\omega(v_g,v_h))^2)^{\tfrac14}\geqslant|v_g-v_h|$. But for any $v_h\in \pi(\mc G)$, an element $h= (v_h,u_h)$ with $u_h = u_g-\tfrac12\omega(v_g,v_h)$ belongs to $\mc G$. Therefore, $d_{CK}(g,\mc G) = \inf\{|v_g-v_h|: v_h\in\pi(\mc G)\} = d_e((v_g,0), \pi(\mc G)\times \{0\}) = d_e(g,\mc G)$.
	\end{proof}

%%%%%%%%%%%%%%%%%%%%%%%%%%%%%%%%%%%%%%%%%%%%%%%%%%%%%%%%%%%%%%%%%%%%%%%%%%%%%%%%%%%%%%%%%%%%%%%%%%%%%%%%%%%%%%%%%%%%%%%%%%%%%%%%%%%%%%%%%%%%%%%%%%%%%%%%%%%%%%%%%%%%
	\section{Expansive Sets and Coding}
	Let us recall, that from now on we assume that the discrete group $H$ acts via homeomorphisms on a compact metric space $(\mathbbm X, \rho)$. We denote the action of an element $g\in H$ on $\x\in\mathbbm X$ by the multiplicative notation $g\x$. For $A\subset \mc H$ and $\mathbbm x,\mathbbm y\in \mathbbm X$, we put $\rho^{A}(\mathbbm x,\mathbbm y) = \sup\{\rho(g\x,g\y):g\in A\cap H\}$. We say that this action is \emph{expansive} if there exists $\eta>0$ such that, for all $\x,\y\in\mathbbm X$, $\rho^{\mc H}(\x,\y) = \rho^{H}(\x,\y)\leqslant\eta$ implies $\x=\y$. We call $\eta$ {an} \emph{expansive constant}. Throughout the rest of the paper, the expansive constant $\eta$ remains fixed.
	
	\begin{definition}
		\label{def_expansive}
		We say that a set $A \subset \mc H$ is \emph{expansive} if there exist $\varepsilon>0$ and $t > 0$ such that $\rho^{A^t}(\x,\y)\leqslant \varepsilon$ implies $\x=\y$. Otherwise, we say that $A$ is \emph{nonexpansive}.
	\end{definition}
	
	\begin{remark}
		Notice that every superset of an expansive set is also expansive; and conversely, every subset of a nonexpansive set is also nonexpansive.    
	\end{remark}
	
	\begin{fact}
		\label{Expan_in_d_iff_expan_in_C-K}
		Since $||\cdot|| = d(\cdot,\mathbf 0)$ is equivalent to $||\cdot||_{CK}$, $A$ is expansive with respect to $d$ if and only if it is expansive with respect to $d_{CK}$. 
	\end{fact}
	
	\begin{lemma}
		\label{SCA}
		Let $\{A_\alpha\}_{\alpha>0}$ be a family of bounded subsets of $\mc H$ such that $A_\alpha\subset A_\beta$ for $\alpha < \beta$, and their union $A:=\bigcup_{\alpha>0}A_\alpha$ is such that $\rho^A(\mathbbm x,\mathbbm y)\leqslant \eta$ implies $\mathbbm x=\mathbbm y$. Then for each $\varepsilon>0$ there is $\alpha>0$ such that $\rho^{A_\alpha}(\mathbbm x,\mathbbm y)\leqslant\eta$ implies $\rho(\mathbbm x,\mathbbm y)\leqslant \varepsilon$.
	\end{lemma}
	\begin{proof}
		Assume that there is $\varepsilon>0$ such that for every $\alpha$ there are $\mathbbm x_\alpha,\mathbbm y_\alpha$ such that $\rho^{A_\alpha}(\mathbbm x_\alpha,\mathbbm y_\alpha)\leqslant \eta$ and $\rho(\mathbbm x_\alpha,\mathbbm y_\alpha)>\varepsilon$. Since for $\alpha<\beta$ we have $A_\alpha\subset A_\beta$, we also have $\rho^{A_\alpha}(\mathbbm x,\mathbbm y)\leqslant \rho^{A_\beta}(\mathbbm x,\mathbbm y)$. In particular, whenever $\alpha<\beta$, we get $\rho^{A_\alpha}(\mathbbm x_\beta,\mathbbm y_\beta)\leqslant \rho^{A_\beta}(\mathbbm x_\beta,\mathbbm y_\beta)\leqslant \eta$.
		
		Since $\mathbbm X$ is compact, by considering only integer alpha and choosing appropriate subsequences, we get $\mathbbm x_{n_k}, \mathbbm y_{n_k}$ convergent to some $\mathbbm x, \mathbbm y$, respectively. By the continuity of the metric, we have $\rho(\mathbbm x, \mathbbm y)\geqslant\varepsilon$. Fix a natural number $m$. The function $\rho^{A_m}(\cdot,\cdot)$ is continuous because $|A_m\cap H|<\infty$ by Proposition \ref{Bounded_are_finite}. Therefore, $\rho^{A_m}(\mathbbm x_{n_k},\mathbbm y_{n_k})$ converges to $\rho^{A_m}(\mathbbm x,\mathbbm y)$. But $\rho^{A_m}(\mathbbm x_{n_k},\mathbbm y_{n_k})\leqslant \eta$ for $n_k>m$, by the argument above. Hence $\rho^{A_m}(\mathbbm x,\mathbbm y)\leqslant \eta$, for any natural $m$, and $\rho^A(\mathbbm x, \mathbbm y) = \sup\{\rho^{A_\alpha}(\mathbbm x,\mathbbm y): \alpha>0\}=\sup\{\rho^{A_m}(\mathbbm x,\mathbbm y): m\in \mathbb N\}\leqslant \eta$. Thus $\mathbbm x= \mathbbm y$, which gives us a contradiction, because $\rho(\mathbbm x, \mathbbm y)\geqslant\varepsilon$.
	\end{proof}

    The following is an analogue of \cite[Lemma 2.3]{BL}    
	\begin{lemma}
     \label{BL_Lemma2.3}
		If $A\subset \mc H$ is expansive, then there is $s>0$ such that $\rho^{A^s}(\mathbbm x,\mathbbm y)\leqslant \eta$ implies $\mathbbm x=\mathbbm y$. We call such $s$ an \emph{expansive radius}.
	\end{lemma}
	\begin{proof}
		Let $t,\varepsilon>0$ be as in Definition \ref{def_expansive} for $A$. By Lemma \ref{SCA}, there is $r>0$ such that $\rho^{B(r)}(\mathbbm x, \mathbbm y)\leqslant \eta$ implies $\rho(\mathbbm x, \mathbbm y)\leqslant \varepsilon$. Take $s = t+r$. Let $\x, \y\in\mathbbm X$ be such that $\rho^{A^s}(\x,\y)\leqslant \eta$. For any $g\in A^t\cap H$ we have $B(r)g\subset A^s$, because 
		if $f\in B(r)$, then for any $h\in A$ we get
		$d(fg,h)=d(f,hg^{-1})\leqslant d(f,\mathbf 0) + d(\mathbf 0, hg^{-1})\leqslant r+d(g, h)$. Thus $\rho^{B(r)g}(\x,\y)=\rho^{B(r)}(g\x,g\y)\leqslant \eta$, so $\rho(g\x,g\y)\leqslant \varepsilon$, for any $g\in A^t\cap H$. That gives us $\rho^{A^t}(\x,\y)\leqslant\varepsilon$, so $\x=\y$.
	\end{proof}
    
	\begin{fact}
		A subgroup $\mc G \subset \mc H$ is expansive if and only if $\textrm{\emph{Span}}(\mc G)$ is expansive.
	\end{fact}
    \begin{proof}
        Since $\mc G\subset \textrm{Span}(\mc G)$, we also have $\mc G^t\subset (\textrm{Span}(\mc G))^t$ for ant $t>0$. Hence, if $\mc G$ is expansive with an expansive radius $t>0$, then $\textrm{Span}(\mc G)$ is also expansive with the same radius. 
        From Lemma \ref{distance_between_G_and_Span(G)}, it follows that there is $s>0$ such that $\textrm{Span}(\mc G)\subset \mc G^s$.
        By Proposition \ref{Thickening_the_thickening}, $(\textrm{Span}(\mc G))^t\subset (\mc G^s)^t \subset \mc G^{t+s}$ for any $t>0$. Therefore, if $\textrm{Span}(\mc G)$ is expansive with radius $t$, then $\mc G$ is expansive with an expansive radius $t+s$.
    \end{proof}

    In \cite{BL}, the authors provide a definition of \emph{coding} one set by another. This relation is invariant under translates by any vectors from $\mathbb R^D$. We provide two notions of coding---an analogue of the original, and its weaker version, which is invariant under translates by elements of the discrete group $H$. This weaker version will allow us to carry out some of our arguments, while maintaining greater simplicity.
    
	\begin{definition}
		Let $A,B$ be subsets of $\mc H$. We say that $A$ \emph{weakly codes} $B$ if $\rho^{A}(\x,\y)\leqslant \eta$ implies $\rho^{B}(\x,\y)\leqslant\eta$ for all $\x,\y\in\mathbbm X$. We say that $A$ \emph{codes} $B$ if for any $g\in\mc H$, $\rho^{Ag}(\x,\y)\leqslant \eta$ implies $\rho^{Bg}(\x,\y)\leqslant\eta$ (i.e., $Ag$ weakly codes $Bg$).
	\end{definition}
    
	\begin{remark}
		Weak coding of $B$ by $A$ is equivalent to the statement: for any $g\in H$, $\rho^{Ag}(\x,\y)\leqslant \eta$ implies $\rho^{Bg}(\x,\y)\leqslant\eta$.
	\end{remark}
    
	\begin{remark}
		\label{coding_of_transformed_sets}
		If $A$ (weakly) codes $B$, then $Ag$ (weakly) codes $Bg$ for any $g\in \mc H$ ($g\in H)$. If $A_\alpha$ (weakly) codes $B_\alpha$, then $\bigcup_\alpha A_\alpha$ (weakly) codes $\bigcup_\alpha B_\alpha$. As a consequence, if $A$ (weakly) codes $B$ and $C\subset \mc H$ ($C\subset H$), then $A\cdot C$ (weakly) codes $B\cdot C$.
	\end{remark}
    \begin{proof}
     The first statement follows directly from the definitions. In order to prove the second one, take two families $\{A_\alpha\}_{\alpha \in \mc A}$ and $\{B_\alpha\}_{\alpha \in \mc A}$ of subsets of $\mc H$. If $A_\alpha$ codes $B_\alpha$ for all $\alpha$ then $A_\alpha g$ weakly codes $B_\alpha g$ for every $g \in \mc H$ and $\alpha \in \mc A$. We have $\left(\bigcup_{\alpha} A_\alpha\right)g = \bigcup_{\alpha} A_\alpha g$, so if $\rho^{\left(\bigcup_{\alpha} A_\alpha\right)g}(\x, \y) \leqslant \eta$, then for every $\alpha$ we have $\rho^{A_\alpha g}(\x, \y)\leqslant \eta$. For each $\alpha$ we then have $\rho^{B_\alpha g}(\x, \y)\leqslant \eta$, so $\rho^{\bigcup_{\alpha} B_\alpha g}(\x, \y) = \rho^{\left(\bigcup_{\alpha} B_\alpha\right)g}(\x, \y) \leqslant \eta$, meaning that $\bigcup_\alpha A_\alpha$ codes $\bigcup_\alpha B_\alpha$. For the argument regarding weakly coding, it suffices to consider $g = \mathbf 0$.

        For the final statement, notice that for $A, C\subset \mc H$ we have $AC = \bigcup_{g\in C}Ag$. If $A$ codes $B$, then, by the first two statements, this union codes $\bigcup_{g\in C}Bg = BC$. If $A$ weakly codes $B$, we must restrict $C$ to subsets of the discrete group $H$, since weak coding of $Bg$ by $Ag$ is guaranteed only for $g\in H$.
    \end{proof}
	
	\begin{remark}
		By Lemma \ref{BL_Lemma2.3}, $A$ is expansive with expansive radius $t>0$ if and only if $A^t$ weakly codes $\mc H$.
	\end{remark}
    
    Let $\lambda >0$ be the constant given by Corollary \ref{d(g,H)<=lambda} (the universal maximal distance between any element of $\mc H$ and the discrete group $H$).     
    \begin{lemma} 
        If $A\subset \mc H$ is expansive with expansive radius $t$ and $A\cdot \mc Z = A$, then $A^{t+\lambda}$ codes $\mc H$.
    \end{lemma}

    \begin{proof}
        Fix $g\in\mc H$. By Corollary \ref{d(g,H)<=lambda} there is $\tilde g\in H$ such that $d(g^{-1},\tilde g^{-1}) \leqslant \lambda$. Take any $h\in A^t \tilde g$. By Lemma \ref{AZ=Z_implies_Ag=gA} we have $Ag_1g_2 = Ag_2g_1$ for any $g_1,g_2\in\mc H$. Therefore
        \begin{align*}
            d(h, Ag) &= d(\mathbf 0, Ag h^{-1}) = d(\mathbf 0,A h^{-1} g ) = d(g^{-1} h, A)\\
            &\leqslant d(g^{-1}h, \tilde g^{-1}h) + d(\tilde g^{-1}h, A) = d(g^{-1}, \tilde g^{-1}) + d(\mathbf 0, A h^{-1}\tilde g) \\
            & = d(g^{-1}, \tilde g^{-1}) + d(\mathbf 0, A\tilde g h^{-1}) = d(g^{-1}, \tilde g^{-1}) + d(h, A\tilde g) \leqslant \lambda+t.
        \end{align*}
        Hence $A^t\tilde g \subset A^{t+\lambda}g$. If $A^t$ weakly codes $\mc H$, so does $A^t\tilde g$, because $\tilde g\in H$. Therefore $A^{t+\lambda}g$ weakly codes $\mc H$, for any $g\in\mc H$.
    \end{proof}

    \begin{corollary}
        A vertical group $\mc G$ is expansive if and only if $\mc G^t$ codes $\mc H$ for some $t>0$.
    \end{corollary}
	
%%%%%%%%%%%%%%%%%%%%%%%%%%%%%%%%%%%%%%%%%%%%%%%%%%%%%%%%%%%%%%%%%%%%%%%%%%%%%%%%%%%%%%%%%%%%%%%%%%%%%%%%%%%%%%%%%%%%%%%%%%%%%%%%%%%%%%%%%%%%%%%%%%%%%%%%%%%%%%%%%%%%
    \section{Expansiveness of Vertical Groups}

    In this section, we focus on describing the expansiveness in the class of vertical groups, which is the main goal for the presented work. In the case of $\mathbb{Z}^D$ actions considered in \cite{BL}, the main interest, as expansive sets,  was on linear subspaces of $\mathbb R^D$. It is obvious in that case, that  if $\mathbbm X$ is infinite, then the trivial space $\{0\}$ is nonexpansive, because we cannot have finite expansive sets (see the proof of Theorem \ref{Z_never_expansive}). The existence of a nonexpansive linear subspace is used to show the existence of $D-1$-dimensional nonexpansive subspace. In our case, however, the smallest vertical subgroup is the vertical axis $\mc Z$. Its nonexpansivness not only allows us to show the existence of other nonexpansive vertical groups---it is a necessary condition for such, as any subset of a nonexpansive set is nonexpansive, and every vertical group contains $\mc Z$. 
    
    \begin{theorem}
		\label{Z_never_expansive}
		If $\mathbbm X$ is infinite, the vertical axis $\mc Z$ is not expansive.
	\end{theorem}
	For $n\in \mathbb N$, we put $I^+_n=\{\mathbf 0, z, \ldots, z^{n}\}$, $I^-_n=\{\mathbf 0, z^{-1}, \ldots, z^{-n}\}$, and $I_n=\{z^{-n}, \ldots, z^{-1}, \mathbf 0, z, \ldots, z^{n}\}$. Since $\mc Z$ is the center of $\mc H$, for any $A\subset \mc H$ we have $I_nA=AI_n$. Similarly, for $t>0$ we define
	$I^{t}_n=\{(v,u)\in\mc Z^{t}\cap H:-n\leqslant u\leqslant n\}$. Note that $I^{t}_n \not = {(I_n)}^t$.
	\begin{lemma}
    \label{ItnIm=Itn+m}
		We have $I^{t}_n\cdot I_m = I^{t}_{n+m}$.
	\end{lemma}
	\begin{proof}
        Let $h = (v_h,u_h)\in I^t_n$ and $g = (0,u)\in I_m$. Then $hg = (v_h, u_h + u)\in H$. We have $d(hg, \mc Z) = d(h,\mc Z g^{-1})= d(h,\mc Z)\leqslant t$ and $|u_h+u|\leqslant n+m$. Hence, $hg\in I^t_{n+m}$.

        Conversely, let $h = (v_h, u_h) \in I^t_{n+m}$. For $g = (0,u)\in I_m$ we have $hg^{-1}=(v_h,u_h-u)$. There is an element $g \in I_m$ such that $|u_h-u|\leqslant n$. Moreover, similarly as above, we have $d(hg^{-1}, \mc Z)\leqslant t$. Thus $hg^{-1}\in I^t_{n}$, so $h\in I^t_n \cdot I_m$.
	\end{proof}
	
	\begin{lemma}
		\label{consequence_of_SCA_for_Z}
		If $\mc Z$ is expansive with expansive radius $t>0$, then for any set $B\subset \mc H$ satisfying $|B\cap H|<\infty$, there is $n$ such that $I^t_n$ weakly codes $B$.
	\end{lemma}
	\begin{proof}
		Since $|B\cap H|<\infty$, there is $\varepsilon>0$ such that $\rho(\x,\y)\leqslant\varepsilon$ implies $\rho^B(\x,\y)\leqslant \eta$. We have $\mc Z^t\cap H=\bigcup_{n\in \mathbb N} I^t_n$ and we conclude by using Lemma \ref{SCA}.
	\end{proof}
	\begin{lemma}
    \label{OdcinekKodujeNadIPodSoba}
		If $\mc Z$ is expansive with expansive radius $t>0$, then there is $N\in \mathbb N$ such that $I^t_N$ weakly codes $I^t_{N+1}$.
	\end{lemma}
	\begin{proof}
	        To slightly simplify the notation in this proof, take $x=x_1,y=y_1$. By Lemma \ref{consequence_of_SCA_for_Z}, there is $n$ such that $I^t_n$ weakly codes $I^t_0 \cdot \{x,x^{-1},y,y^{-1}\}$. Therefore, by Remark \ref{coding_of_transformed_sets} and Lemma \ref{ItnIm=Itn+m}, $I^t_{2n}=I^t_n\cdot I_n$ weakly codes $I^t_0\cdot \{x,x^{-1},y,y^{-1}\}\cdot I_n=I^t_n\cdot \{x,x^{-1},y,y^{-1}\}$, which, in turn, weakly codes $I^t_0\cdot \{x,x^{-1},y,y^{-1}\}\cdot \{x,x^{-1},y,y^{-1}\}$. Inductively, for all $M\in\mathbb N, L\in \{0,1,\ldots, M\}$ and $k_1,\ldots,k_m\in\mathbb Z \text{ such that } |k_1|+\ldots+|k_m|=L$,        
        \begin{equation}
        \label{ItKn_general_coding}
            I^t_{Mn} \text{ weakly codes } I^t_{(M-L)n}x^{k_1}y^{k_2}x^{k_3}\ldots y^{k_m}            .
        \end{equation}        
		In particular, for any $K\in\mathbb N$, the set $I^t_{4Kn}$ weakly codes $I^t_{2Kn}y^kx^ly^{K-k}x^{K-l}$ for any $k,l\in\{0,1,\ldots,K\}$. Using $xy=zyx$, we obtain $y^kx^ly^{K-k}x^{K-l}=y^Kx^Kz^{l(K-k)}$. Hence, for any $k\in\{0,1,\ldots,K^2\}$, the set $I^t_{4Kn}$ weakly codes $I^t_{2Kn}y^Kx^Kz^{k}$. Thus,        
        \begin{equation}
        \label{It4Kn_weakly codes_It2KyK...}
		      I^t_{4Kn}\text{ weakly codes } I^t_{2Kn}y^Kx^KI^+_{K^2},
		\end{equation}         
        which, by \eqref{ItKn_general_coding}, weakly codes 
        $$I_{(2K-2K)n}x^{-K}y^{-K}y^Kx^KI^+_{K^2}=I^t_0I^+_{K^2}.$$
        Therefore, $I^t_{4Kn}$ weakly codes $I^t_0I^+_{K^2}$. On the other hand,         
        \begin{equation*}
            y^Kx^KI^+_{K^2}=x^Ky^Kz^{-K^2}I^+_{K^2}=x^Ky^KI^-_{K^2},
        \end{equation*}
        so by \eqref{It4Kn_weakly codes_It2KyK...} and again \eqref{ItKn_general_coding}, $I^t_{4Kn}$ also weakly codes $I^t_0I^-_{K^2}$. Therefore, by Remark \ref{coding_of_transformed_sets} and Lemma \ref{ItnIm=Itn+m}
        \begin{equation*}
            I^t_{4Kn} \text{ weakly codes } I^t_0I^+_{K^2} \cup I^t_0I^-_{K^2} = I^t_{0}I_{K^2}=I^t_{K^2}.
        \end{equation*}
        Since $K$ was arbitrary, we can take $K\in\mathbb N$ such that $K^2\geqslant 4Kn+1$. Since $I^t_{4Kn}$ weakly codes $I^t_{K^2}$, in particular it weakly codes $I^t_{4Kn+1}$. Take $N=4Kn$.
	\end{proof}
	\begin{proof}[Proof of Theorem \ref{Z_never_expansive}]
		Assume that $\mc Z$ is expansive with expansive radius $t>0$. By Lemma \ref{OdcinekKodujeNadIPodSoba}, there is some $N$ such that $I^t_N$ weakly codes $I^t_{N+1}=I^t_N\cdot I_1$, which weakly codes $I^t_{N+1}\cdot I_1=I^t_{N+2}$. Inductively, $I^t_N$ weakly codes $I^t_{n}$ for arbitrary large $n$. Hence $I^t_N$ weakly codes $\mc Z^t$, because $\bigcup_{n\in\mathbb N}I^t_n=\mc Z^t\cap H$. Therefore $I^t_N$ weakly codes $\mc H$, so $\rho^{I^t_N}(\x,\y)\leqslant \eta$ implies $\x=\y$. This means that sets of the form $\{\y\in\mathbbm X:\sup_{g\in I^t_N}\rho(g\x,g\y) < \eta\}$ are singletons. But these sets are open, because $I^t_N$ is a finite subset of $H$, and $H$ acts on $\mathbbm X$ via homeomorphisms. Since $\mathbbm X$ is compact, it must be finite.
	\end{proof}
	
	For a vertical group $\mc G$, by $\mc G^t_e$ we denote the thickening of $\mc G$ by $t\geqslant 0$ with respect to the Euclidean metric $d_e$. We say that $\mc G$ is \emph{expansive with respect to $d_e$} if, for some $t\geqslant 0$, the thickening $\mc G^t_e$ codes $\mc H$.
	
	\begin{lemma}
		\label{For_vertical_euclidean_is_good}
		Let $\mc G$ be a vertical group. Then $\mc G$ is expansive with respect to $d$ if and only if it is expansive with respect to $d_e$.
	\end{lemma}
	\begin{proof}
		By Proposition \ref{Expan_in_d_iff_expan_in_C-K}, all we need to show is that $\mc G$ is expansive with respect to $d_e$ if and only if it is expansive with respect to $d_{CK}$. But that follows directly from Lemma \ref{Thickening_vertical_group}
	\end{proof}

    From now on, we aim to mimic the techniques presented in \cite{BL}. The main novelty is that we use the infinite versions of Boyle's and Lind's rectangles $V^t(r)$, namely, we take the products of these rectangles with the vertical axis. That affects the proof of Lemma \ref{BL_Lemma3.2} (a replacement for \cite[Lemma 3.2]{BL}), where we cannot use the compactness argument directly, but we need to bound our `rectangles' first, use the argument, and then return to the unbounded versions. In this way, the reasoning is reduced to the Euclidean geometry arguments. The proofs of the rest of the results strongly resemble those presented in \cite{BL}, but we provide them for completeness.
	
	For a linear subspace $V< \mathbb R^{2D}$, let $\pi_V$ be the projection onto $V$, that is, $\pi_V((v,u))$ is the orthogonal projection of $v$ onto $V$ in the space $\mathbb R^{2D}$. For any $r\geqslant 0$ we define $\mc G^t_e(r) = \{g\in\mc G^t_e: |\pi_{V}(g)|\leqslant r\}$, where $V=\pi(\mc G)$ (i.e., $\mc G = V\times \mathbb R$). In the language of \cite{BL}, when $\mc G= V\times \mathbb R$, we have $\mc G^t_e(r) = V^t(r)\times \mathbb R$, where $V^t(r) = \{v\in \mathbb R^{2D}: d_e(v,V) \leqslant t \text{ and } d_e(v, V^{\perp}) \leqslant r\}$. Equivalently, $V^t(r) = \pi(\mc G_e^t(r))$. We have
        \begin{equation}
        \label{eq:G^t_1_e(r_1)G^t_2_e(r_2)=G^t_1+t_2_e(r_1 + r_2)}
            \mc G^{t_1}_e(r_1)\cdot \mc G^{t_2}_e(r_2) = \mc G^{t_1 + t_2}_e(r_1 + r_2)
        \end{equation} 
    
    \begin{remark}
    \label{Rem:G_e^t(r)_contain_Z}
        Since $\mc G_e^t(r)\cdot\mc Z = \mc G_e^t(r)$, for any $v\in\mathbb R^{2D}$ and $u,\tilde u\in\mathbb R$ we have
        \begin{equation}
        \label{eq:G_e^t(r)(v,u)=G_e^t(r)(v,0)}
            \mc G_e^t(r)(v,u) = \mc G_e^t(r)(v,\tilde u) = \mc G_e^t(r)(v,0).
        \end{equation}
    \end{remark}
	
	\begin{lemma}
		\label{BL_Lemma3.2} 
		If a vertical subgroup $\mc G$ is expansive, then there is $t>0$ with the property that for every $s>0$ there is $r>0$ such that $\mc G^t_e(r)$ codes $\mc G^s_e(0)$. Hence $\mc G^t_e(r+a)$ codes $\mc G^s_e(a)$ for all $a>0$.
	\end{lemma}
	
	\begin{proof}
		Let us recall that $\eta$ is the expansive constant and that $Z = \langle \{z\} \rangle$. Let $\mc G$ be expansive with expansive radius $t_0>0$. Take $t=t_0+D$ and fix $s>0$. Because the discrete Heisenberg group $H$ acts by homeomorphisms, we can find $\varepsilon>0$ such that $\rho(\x,\y)\leqslant\varepsilon$ implies $\rho^{B_e(C)}(\x,\y)\leqslant \eta$, where $B_e(C)$ is the closed ball with respect to the Euclidean metric $d_e$, centered at $\mathbf 0$ and with radius $C=s+D+1$. Notice that such $C$ guarantees the inclusion $B_e(s+D) \mc Z \subset B_e(C) Z$. To see this, first observe that $B_e(s+D) \mc Z =  \{v \in \mathbb{R}^{2D}: |v| \leqslant s+D\}\times \mathbb{R}$. Hence for each $(v,u)\in B_e(s+D) \mc Z$ we have $d_e((v,u), (0,[u])) \leqslant |v| + |u-[u]|\leqslant s+D+1$, where $[u]$ denotes the integer part of $u$.
		
		We need ``bounded versions'' of $\mc G_e^{t_0}(r)$ sets, namely ${\mc{B}_e^{t_0}(r)} = \mc G^{t_0}_e(r)\cap B_e(r)$. These sets are bounded and $\bigcup_{r>0}{\mc{B}_e^{t_0}(r)} = \mc G^{t_0}_e$, so by Lemma \ref{SCA}, there exists $r>D$ such that $\rho^{{\mc{B}_e^{t_0}(r-D)}}(\x,\y)\leqslant\eta$ implies $\rho(\x,\y)\leqslant \varepsilon$. Take such $r$. 
		
		Fix $g = (v,u) \in \mc H$ and $\x,\y\in\mathbbm X$ such that $\rho^{\mc G_e^t(r)g}(\x,\y)\leqslant \eta$. Let $\tilde v\in\mathbb Z^{2D}$ be such that $|v-\tilde v|<D$ (note that there exists $\tilde v\in\mathbb Z^{2D}$ such that $|v-\tilde v|\leqslant\frac{\sqrt{2D}}{2}$) and define $\tilde g = (\tilde v, 0)\in H$. Using \eqref{eq:G_e^t(r)(v,u)=G_e^t(r)(v,0)}, we have $\mc G_e^{t_0}(r-D)\tilde g \subset \mc G_e^t(r)g$. Since ${\mc{B}_e^{t_0}(r-D)}\subset \mc G_e^{t_0}(r-D)$, we have $\rho^{{\mc{B}_e^{t_0}(r-D)\tilde g}}(\x,\y) = \rho^{{\mc{B}_e^{t_0}(r-D)}}(\tilde g\x,\tilde g\y) \leqslant \eta$, and hence $\rho(\tilde g\x, \tilde g\y)\leqslant\varepsilon$. By earlier assumptions, we have $\rho^{B_e(C)}(\tilde g\x, \tilde g\y) = \rho^{B_e(C)\tilde g}(\x, \y)\leqslant \eta$. Hence $\mc G_e^t(r)g$ codes $B_e(C)\tilde g$, and so, by Remark \ref{coding_of_transformed_sets}, $\mc G_e^t(r)g = \mc G_e^t(r)Zg = \mc G_e^t(r)gZ$ codes $B_e(C)\tilde gZ = B_e(C)Z\tilde g$. But 
		\begin{equation*}
			B_e(C)Z\tilde g \supset B_e(s+D)\mc Z\tilde g \supset B_e(s)\mc Z g,
		\end{equation*}
		hence $\mc G_e^t(r)g$ codes $B_e(s)\mc Z g$, which contains $\mc G_e^s(0)g$. Since $g$ was arbitrary, $\mc G_e^t(r)$ codes $\mc G_e^s(0)$. Finally, by \ref{eq:G^t_1_e(r_1)G^t_2_e(r_2)=G^t_1+t_2_e(r_1 + r_2)} and Remark \ref{coding_of_transformed_sets}, $\mc G_e^t(r+a) = \mc G_e^t(r) \cdot \mc G_e^0(a)$ codes $\mc G_e^s(0)\cdot \mc G_e^0(a) = \mc G_e^s(a)$.
	\end{proof}	

    \begin{lemma}
	\label{BL_Lemma3.3} 
        If $\mc G^t_e(r)$ codes $\mc G^{t+\varepsilon}_e(0)$ for some $t,r,\varepsilon >0$, then $\mc G$ is expansive.
    \end{lemma}

    \begin{proof}
        If $\mc G_e^t(r)$ codes $\mc G_e^{t+\varepsilon}(0)$, then by \eqref{eq:G^t_1_e(r_1)G^t_2_e(r_2)=G^t_1+t_2_e(r_1 + r_2)} and Remark \ref{coding_of_transformed_sets}, $\mc G_e^t(r+a) = G_e^t(r)\cdot G_e^0(a)$ codes $\mc G_e^{t+\varepsilon}(0)\cdot\mc G_e^0(a) = \mc G_e^{t+\varepsilon}(a)$, for every $a>0$. Since $\mc G_e^s = \bigcup_{a>0}\mc G_e^s(a)$ for any $s$,  $\mc G_e^t $ codes $\mc G_e^{t+\varepsilon}$, again by Remark \ref{coding_of_transformed_sets}. On the other hand, $\mc G_e^{t+\varepsilon} = \mc G_e^{t}\cdot \mc G_e^\varepsilon$, so it codes $\mc G_e^{t+\varepsilon}\cdot \mc G_e^\varepsilon = \mc G_e^{t+2\varepsilon}$. Inductively, $\mc G_e^{t}$ codes $\bigcup_{a>0} \mc G_e^{t+a} = \mc H$.
    \end{proof}

    We call a vertical group \emph{$k$-dimensional} if it has dimension $k$ as a linear space, i.e., it has the form $V\times \mathbb R$, where $V$ has dimension $k-1$. Let $\mathbb G_{k}$ be the Grassmann manifold of $k$-dimensional linear subspaces of $\mathbb R^{2D}$.
    
	\begin{definition}
		For $k\leqslant 2D-1$ we define:
		\begin{itemize}
			\item $\mathbb E_k = \{V\in \mathbb G_k: \mc G = V\times \mathbb R\textrm{ is expansive}\}$;
			\item $\mathbb N_k = \{ V \in \mathbb G_k: \mc G = V \times \mathbb R \textrm{ is nonexpansive}\}$.            
		\end{itemize}
	\end{definition}
	
	\begin{lemma}
		\label{Just_like_BL}\,
		\begin{enumerate}
			\item \label{BL_Lemma3.4} Let $V \in \mathbb E_k$. Then there are $r = r_{V} , t = t_{V} $, and a neighborhood $\standardcal N_{V}$ of $V$ in $\mathbb G_k$ such that, for every $W\in \standardcal N_{V}$, for $\mc F = W \times \mathbb R$ we have that $\mc F_e^t(r)$ codes $\mc F^{t+1}_e(0)$. Hence $\standardcal N_{V} \subset \mathbb E_k$, so $\mathbb E_k$ is open in $\mathbb G_k$.
			\item \label{BL_Lemma3.5} Let $\standardcal K$ be a compact subset of $\mathbb E_k$. Then there are $r_{\standardcal K} > 0$ and $t_{\standardcal K} > 0$ such that $\mc G^t_e(r)$ codes $\mc G^{t+1}_e(0)$ for every $\mc G = V \times \mathbb R$ such that $V \in \standardcal K$.
		\end{enumerate}
	\end{lemma}
	\begin{proof}		
		(\ref{BL_Lemma3.4})  Let $\mc G = V\times \mathbb R$. By Lemma \ref{BL_Lemma3.2}, there are $t_0 >0$ and $r_0>0$ such that $\mc G_e^{t_0}(r_0)$ codes $\mc G^{t_0 + 3}(1)$. Take $t = t_0 + 1$ and $r = r_0 + 1$. For a $k$-dimensional subspace $W<\mathbb R^{2D}$, if $W$ is sufficiently close to $V$, it satisfies the inclusions $V^{t_0}(r_0) \subset W^t(r)$ and $W^{t+1}(0) \subset V^{t_0+3}(1)$. For each such $W$ we take $\mc F = W\times \mathbb R$. The appropriate inclusions hold, hence $\mc F_e^t(r)$ codes $\mc F_e^{t+1}(0)$.
		
		(\ref{BL_Lemma3.5}) For each $V \in \standardcal K$ take $\standardcal N_{V}$, $t_{V}$ and $r_{V}$ as in \ref{BL_Lemma3.4}. Since $\standardcal K$ is compact, it has a finite cover $\{\standardcal N _{V_1}, \ldots , \standardcal N_{V_l}\}$. Let $\mc G_i = V_i \times \mathbb R$. Notice that if $\mc G_e^{t_0}(r_0)$ codes $\mc G^{t_0+1}(0)$, then for every $t>t_0$ and $r>r_0$, $\mc G_e^{t}(r)$ codes $\mc G^{t+1}(0)$. Hence, we can take $t_{\standardcal K} = \max\{t_{V_1}, \ldots, t_{V_l}\}$ and $r_{\standardcal K} = \max\{r_{V_1}, \ldots, r_{V_l}\}$.
	\end{proof}
	
	\begin{theorem}
		\label{BL_Thm3.6} 
		If $\mc G = V \times \mathbb R$ is a nonexpansive vertical group of dimension $k\leqslant 2D-1$, then there is a $2D$-dimensional nonexpansive vertical group containing $\mc G$.
	\end{theorem}
	\begin{proof}
		Let $\standardcal K = \{ W \in \mathbb G_{2D-1}: V \subset W\}$. 
        Assume that for every $W \in \standardcal K$, $\mc F = W\times \mathbb R$ is expansive. The set $\standardcal K$ is compact in $\mathbb G_k$, so we can choose $t = t_{\standardcal K}$ and $r = r_{\standardcal K}$ as in (\ref{BL_Lemma3.5}). 
		
		Let $T_0 > 0$ be sufficiently big to meet the following: for each $v \in V^{T_0 + \frac 12}$, there is $W \in \standardcal K$ and $\tilde v \in \mathbb R^{2D}$ such that $W^{t}(r) + \tilde v \subset V^{T_0}$ and $v \in W^{t+1}(0) + \tilde v$. Since $\mc F_e^{t}(r) = W^t(r) \times \mathbb R$ codes $\mc F_e^{t+1}(0)$, we have that $\mc G_e^{T_0}$ codes $\mc G_e^{T_0+\frac12}$. By induction and using Remark \ref{coding_of_transformed_sets}, similarly to the proof of Lemma \ref{BL_Lemma3.3}, we obtain that $\mc G_e^{T_0}$ codes $\mc H$, which contradicts its nonexpansiveness.
	\end{proof}
	
	\begin{corollary}
		Since every subset of a nonexpansive set is nonexpansive, Theorem \ref{BL_Thm3.6} states that the family of $2D$-dimensional vertical groups determines all nonexpansive vertical groups in $\mc H$. Precisely: a vertical group $\mc G$ is nonexpansive if and only if it is contained in some $2D$-dimensional nonexpansive vertical group.
	\end{corollary}
	\begin{theorem}    \label{nonexp2D}
		If $\mathbbm X$ is infinite, there is always at least one nonexpansive $2D$-dimensional vertical group. 
	\end{theorem}
	\begin{proof}
		By Theorem \ref{BL_Thm3.6} we only need the existence of some nonexpansive vertical group. That is guaranteed by Theorem \ref{Z_never_expansive} ($\mc Z$ is the smallest vertical group).
	\end{proof}

%%%%%%%%%%%%%%%%%%%%%%%%%%%%%%%%%%%%%%%%%%%%%%%%%%%%%%%%%%%%%%%%%%%%%%%%%%%%%%%%%%%%%%%%%%%%%%%%%%%%%%%%%%%%%%%%%%%%%%%%%%%%%%%%%%%%%%%%%%%%%%%%%%%%%%%%%%%%%%%%%%%%    
\end{document}